\documentclass[psamsfonts]{proc-l}    % Nicer than default article style:  less
\usepackage{graphicx}
\usepackage[cmtip,all]{xy}
\usepackage{amsmath,amsthm}     % Handy math stuff, theorem environments.
\usepackage{amssymb}            % Fancy math symbols.
\usepackage{euscript}           % Nice script font.
\usepackage{enumerate,calc}
\usepackage{amscd}
\numberwithin{equation}{section}

\newtheorem{thm}[equation]{Theorem}
\newtheorem{defn}[equation]{Definition}
\newtheorem{prop}[equation]{Proposition}

\newtheorem{cor}[equation]{Corollary}
\newtheorem{lemma}[equation]{Lemma}

\theoremstyle{definition}  % Bold headings and Roman body text.

\newtheorem{note}[equation]{Note}

\newtheorem{remark}[equation]{Remark}

  {\end{list}}

\newcommand{\dfn}{\textbf} % Make defined words bold.

\newcommand{\cat}{\EuScript}    % Use \EuScript to name a category.
\newcommand{\cC}{{\cat C}}

\newcommand{\cM}{{\cat M}}
\newcommand{\cN}{{\cat N}}

\newcommand{\cV}{{\cat V}}
\newcommand{\cU}{{\cat U}}

\newcommand{\Gr}{{\cat Grpd}}

\newcommand{\Set}{{\cat Set}}

\newcommand{\field}[1]  {\mathbb #1} % Use blackboard bold for these sets

\newcommand{\Z}         {\field Z}

\DeclareMathOperator*{\colim}{colim}
\DeclareMathOperator*{\holim}{holim}

\DeclareMathOperator{\Hom}{Hom}

\DeclareMathOperator{\diag}{diag}

\DeclareMathOperator{\Iso}{Iso}

\DeclareMathOperator{\Aff}{Aff}

\newcommand{\ra}{\rightarrow}                   % right arrow
\newcommand{\lra}{\longrightarrow}              % long right arrow
\newcommand{\lla}{\longleftarrow}               % long left arrow
\newcommand{\llra}[1]{\stackrel{#1}{\lra}}      % labeled long right
\newcommand{\llla}[1]{\stackrel{#1}{\lla}}      % labeled long right
\newcommand{\we}{\llra{\sim}}                   % weak equivalence

\newcommand{\fib}{\twoheadrightarrow}           % fibration

\newcommand{\dbra}{\rightrightarrows}           % double arrow for eqlizer

\newcommand{\m}{\cM}         %a stack is called m. lower case??
\newcommand{\n}{\cN}         %another stack is called n
\newcommand{\tuborg}{\left\{\begin{array}{ll}}
\newcommand{\sluttuborg}{\end{array}\right.}

\begin{document}

\email{sjh@math.huji.ac.il+}

\title{Characterizing algebraic stacks}

\author{Sharon Hollander}

\address{Department of Mathematics,
Hebrew University, Jerusalem, Israel}

\date{April, 25, 2006}

\begin{abstract}
We extend the notion of algebraic stack to
an arbitrary subcanonical site $\cC$. If the topology on $\cC$ is local 
on the target and satisfies descent for morphisms, we
show that algebraic stacks are precisely those which are 
weakly equivalent to representable presheaves of groupoids
whose domain map is a cover. This leads naturally to a definition of
algebraic $n$-stacks. We also compare different
sites naturally associated to a stack.
\end{abstract}

\maketitle
\bibliographystyle{amsplain}

\section{Introduction}

Stacks arise naturally in the study of moduli problems in
geometry. They were introduced by Giraud \cite{Gi} and Grothendieck, and were 
used by Deligne and Mumford \cite{DM} to study the the moduli spaces of curves.
They have recently become important also in differential 
geometry \cite{Bry} and homotopy theory \cite{G}. 
Higher order generalizations of stacks are also receiving 
much attention from algebraic geometers and homotopy theorists.

In this paper, we continue the study of stacks
from the point of view of homotopy theory started in \cite{H,H2}.
The aim of these papers is to show that many properties of stacks and
classical constructions with stacks are homotopy theoretic in nature.
This homotopy theoretical understanding gives rise to a simpler
and more powerful general theory.
In \cite{H} we introduced model category structures
on different ambient categories in which stacks are the fibrant objects,
and showed that they are all Quillen equivalent. 
In this paper we work with the simplest such model: 
the local model structure on presheaves of groupoids on site $\cC$,
which we denote by $P(\cC,\Gr)_L$. 
%The stacks are the fibrant objects
%in $P(\cC,\Gr)_L$. Equivalently, they are the presheaves of groupoids 
%satisfying
%descent for covers.

Deligne and Mumford introduced the notion of an algebraic stack in 
\cite[Definition $4.6$]{DM}.
This definition generalizes easily to an arbitrary site $\cC$ and our main
result is a characterization of these (generalized) algebraic stacks on sites
satisfying certain mild hypotheses.

A key observation is that the (2-category) fiber product (see Definition \ref{fiber-product}) is 
a model for the homotopy pullback in the model category
 $P(\cC,\Gr)_L$ and this allows us to rewrite the definition of \emph{representable morphism}
in the following homotopy invariant fashion: 
\[ f: \m \to \n  \in P(\cC,\Gr) \]
is representable if for each $X \to \n$ with $X \in \cC$, the homotopy pullback
$\m \times_\n^h X$ is weakly equivalent to a representable.

Generalizing \cite[Definition $4.4$]{DM}, we say that a presheaf of groupoids $\m$ on $\cC$ is
\emph{algebraic} if the diagonal $ \m \to \m \times \m$ is a representable morphism and there
exists a \emph{cover} $X \to \m$ with $X \in \cC$. By cover we mean 
a representable morphism such that for all $Y \in \cC$, the homotopy pullback
\[ X \times^h_\m Y \to Y \]
is weakly equivalent to a cover in $\cC$. 

We say that the (basis for the) topology on $\cC$ is \emph{local} if the notion
of cover is local on the target (Definition \ref{localontarget}).
This condition is satisfied by virtually
all the topologies in use in algebraic geometry and one can always
saturate a basis for a topology so that this condition is satisfied.
A topology on $\cC$ satisfies \emph{descent for morphisms} if the 
contravariant
assignment $X \mapsto \Iso(\cC/X)$ is a stack (Definition \ref{descformorph}).
Our main result is then the following.

\begin{thm}[Theorem \ref{char-alg-stacks}]
Let $\cC$ be a Grothendieck topology which is local on the target 
and satisfies descent for morphisms.
$\m \in P(\cC,\Gr)$ is algebraic if and only if $\m$ is weakly equivalent
in $P(\cC,\Gr)_L$ to a representable presheaf of groupoids $(X_o,X_m)$
with the domain map $X_m \to X_o$ a cover in $\cC$.
\end{thm}
In particular, if $\cC=\Aff_{flat}$ 
(fpqc for affine schemes), this theorem characterizes
algebraic stacks as those weakly equivalent to flat Hopf algebroids.

This result leads naturally to a definition of {\it algebraic $\infty$-stacks
($n$-stacks)}; they are those presheaves of simplicial sets on $\cC$ which are
weakly equivalent in $P(\cC,s\Set)_L$ (see \cite{DHI}) to 
(the $n$-coskeleton of) 
a simplicial object in $\cC$ where all the boundary maps are covers.

In the appendix we consider several natural sites associated to a stack and compare them.
The first is the classical site $\cC/\m$ (see \cite{DM}). 
In this topology, the objects are maps $X \to \m$ with $X \in \cC$.
It is natural to ask for a topology on the over category of $\m$ (in which $\m$
itself is an object). We use the notion of representability to construct a larger 
site $Rep/\m$ and prove that for $\m$ algebraic the sheaves on the 
two sites agree.

We also construct a topology on $P(\cC,\Gr)/\m$ where the covers
are collections of fibrations $\cU_i \to \m$ such that the canonical map 
from the realization of the nerve $|\cU_\bullet| \to \m$
is a weak equivalence. We characterize these covers as those sets of maps whose image is locally a covering
sieve for the topology on $\cC$ (see Proposition \ref{charac-covers}).

\subsection{Relation to other work}  In his 2004 Northwestern thesis 
E. Pribble \cite{P} constructs an equivalence of 2-categories between 
flat Hopf algebroids and rigidified algebraic stacks.
This is essentially equivalent to Theorem \ref{char-alg-stacks} in case 
$\cC$ is affine schemes in the flat topology.

\subsection{Acknowledgments} 
I would like to thank G. Granja for helpful comments.

\subsection{Notation and conventions}

We assume that our fixed base site $\cC$ is small and closed under finite products
and pullbacks. By topology we mean what is usually called a basis for a topology 
\cite[Definition III.2.2]{MM}. We assume the topology is subcanonical, i.e. that 
the representable functors are sheaves, and identify the objects in $\cC$ with 
the sheaves they represent.

We write $P(\cC,\Gr)$ for the category of presheaves of groupoids on $\cC$. 
If $\{U_i \to X\}$ is a cover, we write $U=\coprod_i U_i$ for the coproduct of the
sheaves and $U_\bullet$ for the \emph{nerve of the cover} which 
is the simplicial object obtained by taking iterated fiber products over $X$.
We will sometimes abuse notation and write a cover as $U \to X$.
$|U_\bullet|$ will denote the geometric realization of the simplicial object
in $P(\cC, \Gr)$. Recall that the geometric realization of a simplicial diagram 
$F_\bullet$ in $P(\cC,\Gr)$ is defined by $|F_\bullet|(X) = |F_\bullet(X)|$ 
(see \cite[Section 2.2]{H}).

We will write $P(\cC,\Gr)$ for the category of presheaves of groupoids with 
the \emph{levelwise model structure} where a map $F \to F'$ is a fibration 
(weak equivalence) if and only if $F(X) \to F'(X)$ is a fibration 
(weak equivalence) in $\Gr$ for all $X \in \cC$. 
We will write $P(\cC,\Gr)_L$ for the \emph{local model structure} which is the localization
of $P(\cC,\Gr)$ with respect to the maps $|U_\bullet| \to X$ where $U\to X$ is a cover (see \cite{H}).  The local model structure $P(\cC,\Gr)_L$ is our default.

We will use repeatedly the basic result \cite[Theorem 5.7]{H} which characterizes the
weak equivalences in $P(\cC,\Gr)_L$ as those satisfying the
\emph{local lifting conditions}. 
\begin{defn} \cite[Definition 5.6]{H}
\label{locallifting}
A map $F \to G \in P(\cC,\Gr)$ satisfies the \emph{local lifting conditions} if
\begin{enumerate}
\item Given a commutative square 
\[\parbox[c]{\fill}{
\xymatrix{ \emptyset \ar[r] \ar[d] & F(X) \ar[d] \\ \star \ar[r] & G(X)}}
\hspace{.4cm}
\parbox[c]{\fill}{ \mbox{ $\Rightarrow$ $\exists$ cover $U \ra X$,} }
\hspace{3.3cm}
\parbox[c]{\fill}{
\xymatrix{\star \ar@/^3ex/@{-->}[rrr] \ar[d] & \emptyset \ar[l] \ar[r]
\ar[d] &  F(X) \ar[d] \ar[r] & F(U) \ar[d] \\
\Delta^1 \ar@/_3ex/@{-->}[rrr] & \star \ar[l] \ar[r] & G(X)\ar[r] &
G(U).} }
\]
\\
\item
For $A \ra B$, one of the generating cofibrations 
$\partial \Delta^1 \ra \Delta^1, \text{  } B\Z \ra \star,$ 
given a commutative square
\[
\parbox[c]{\fill}{
\xymatrix{ A \ar[r] \ar[d] & F(X) \ar[d] \\ B \ar[r] & G(X)}}
\hspace{.4cm}
\parbox[c]{\fill}{ \mbox{ $\Rightarrow$ $\exists$ cover $U \ra X$,} }
\hspace{3.3cm}
\parbox[c]{\fill}{
\xymatrix{A  \ar[r] \ar[d] &  F(X) \ar[d] \ar[r] & F(U) \ar[d] \\
B \ar[r] \ar@{-->}[rru] & G(X) \ar[r] & G(U).} }
\]
\end{enumerate}
\end{defn}

\section{Fiber Product}

In this section we will review the {\it fiber product of stacks} \cite[Definition 2.2.2]{LM-B}
from our homotopy theoretic point of view.

\begin{defn} 
\label{fiber-product}
Let 
\[ \cM_1 \xrightarrow{i} \cN \xleftarrow{j} \cM_2 \]
be a diagram in $P(\cC,\Gr)$.  The \emph{homotopy fiber product} 
$\cM_1 \times^h_{\cN} \cM_2$, is the presheaf of groupoids defined as follows:
\begin{enumerate}
\item the objects of $(\cM_1 \times^h_{\cN}\cM_2)(X)$ are triples 
		$(a,b,\phi)$ with $a \in \cM_1(X), b\in \cM_2(X)$ and 
		an isomorphism $\phi:i(a) \we j(b)$, and
\item morphisms of $(\cM_1 \times^h_{\cN}\cM_2)(X)$ from $(a,b,\phi)$ to 
		$(a',b', \phi')$ are pairs $(\alpha, \beta)$ where 
		$\alpha:a \cong a'$ and $\beta: b \cong b'$, such that 
		$\phi' \circ i(\alpha) = j(\beta) \circ \phi$.
\end{enumerate} 
\end{defn}
There are natural projections $p_i:\cM_1 \times^h_{\cN} \cM_2 \ra \cM_i$
and natural homotopy $i \circ p_1 \ra j \circ p_2$ which are 
universal in the following sense.
%The homotopy fiber product satisfies the following universal property.
To give a map $f:\m \ra  \cM_1 \times^h_{\cN}\cM_2$ is the same 
as to give a maps $f_i:\m \ra \m_i$ and a levelwise homotopy 
$i\circ f_1 \ra j\circ f_2$. 

The homotopy fiber product defined above is obviously the homotopy limit of 
the pullback diagram in the category $P(\cC,\Gr)$ with the levelwise model structure.
In fact, it also provides a model for the homotopy pullback in the local 
model structure as we now see.
\begin{lemma}  
\label{levelwisefp}
The homotopy fiber product of Definition \ref{fiber-product} 
is a model for the homotopy pullback in $P(\cC,\Gr)_L$. 
\end{lemma}
\begin{proof}
Consider the pullback diagram in Definition \ref{fiber-product}.
Since $P(\cC,\Gr)_L$ is right proper \cite[Corollary 5.8]{H} (and $P(\cC,\Gr)$ is obviously right proper), 
the homotopy fiber product in both of these model categories is obtained by replacing the map 
$\m_2 \to \n$ by a fibration and taking the pullback.

Factor $\cM_2 \ra \cN$ into a trivial cofibration 
followed by a fibration $\cM_2 \we \cM' \fib \cN$ in $P(\cC,\Gr)_L$.
Further factor $\cM_2 \we \cM'$ into a levelwise trivial cofibration and a levelwise
fibration $\cM_2 \we \cM'' \fib \cM'$. Then we have a levelwise weak equivalence 
$$\m_1 \times^{h}_{\cN} \cM_2 \simeq \m_1 \times_\cN \cM''.$$
The map $\m_1 \times_\cN \cM'' \ra \m_1 \times_\cN \cM'$ is the pullback of a levelwise
fibration and weak equivalence and hence, by \cite[Corollary 5.8]{H}, it is 
itself a weak equivalence.
\end{proof}

\begin{remark}
Since homotopy limits commute with each other,
if $\m_1,\m_2$ and $\n$ in Definition \ref{fiber-product} are stacks,
(presheaves of groupoids satisfying the homotopy sheaf condition, 
see \cite[Definition 1.3]{H})
the homotopy fiber product $\m_1 \times^h_\n \m_2$ is also a stack
and agrees with what is usually called the fiber product of stacks \cite[2.2.2]{LM-B}.
\end{remark}

Given a groupoid object $(X_{0},X_{1})$ in $\cC$ we abuse notation
and let $(X_{0},X_{1})$ denote the presheaf of groupoids of which $X_0$ represents
the objects and $X_m$ represents the morphisms.  We let $\m_{(X_0,X_1)}$
denote the fibrant replacement in $P(\cC, \Gr)_L$ of $(X_0,X_1)$, that is its 
{\it stackification}. 

\begin{lemma}\label{ass-stack}
Let $\m$ be a presheaf of groupoids,
$(X_{0},X_{1})$ be a groupoid object in $\cC$, and 
$(X_{0},X_{1}) \ra \m$ a weak equivalence in $P(\cC,\Gr)_L$.
The map \[ \xymatrix{ X_{1}\ar[r] & X_{0} \times^h_{\m} X_{0} }\]
induced by the domain and range is a weak equivalence. If $\m$ is a stack,
it is a levelwise weak equivalence.
\end{lemma}
\begin{proof}
First we prove that the map
is a weak equivalence for $\m= \m_{(X_0,X_1)}$. By Lemma \ref{levelwisefp}, we need to verify the 
local lifting conditions of Definition \ref{locallifting} for the map $X_1 \to X_0\times^{h}_\m X_0$.
By definition the map $(X_{0},X_{1}) \ra \m$ is a weak equivalence
and so by \ref{locallifting}(2) given two objects $a,b \in X_0(Y)$ and an
isomorphism between their images in $\m(Y)$ there exists a cover 
$U \ra Y \in \cC$
such that this isomorphism lifts to $X_1(U)$.  This implies that  
condition \ref{locallifting}(1) holds for the map $X_1 \ra X_{0} \times^{h}_{\m} X_{0}$.

Similarly, given $\phi_1,\phi_2 \in X_1(X)$, an isomorphism between their images
in $(X_{0} \times^{h}_{\m} X_{0})(X)$ is necessarily trivial (as $X_0(X)$ is discrete)
and so the images of $\phi_1$ and $\phi_2$ in $\m(X)$ are the same.

The fact that $(X_{0},X_{1}) \ra \m$ satisfies
condition \ref{locallifting}(2) for the cofibration $B\Z \to *$  
guarantees the existence of the cover $U$ of $X$
such that $\phi_1$ and $\phi_2$ in $X_1(U)$. This proves one half of \ref{locallifting}(2) for the 
map $X_1 \to X_0\times^{h}_\m X_0$ and the other half is automatic as $X_1(X)$ is
discrete.

For general $\m$, the fact that the map is a weak equivalence follows from the homotopy invariance of 
the homotopy fiber product (Lemma \ref{levelwisefp}).

Since weak equivalences between fibrant objects in $P(\cC,\Gr)_L$ are levelwise weak equivalences, if 
$\m$ is a stack, the map is a levelwise weak equivalence.
\end{proof}

\begin{remark}
\label{incfullsubcat}
If $\m$ is a stack, the statement that $X_1 \to X_0 \times^{h}_\m X_0$ is a
levelwise weak equivalence means that evaluating at each $X \in \cC$
$$ (X_0,X_1)(X) \to \m(X)  $$
is bijective on $\Hom$ sets and that two objects with the same image in $\m(X)$ 
are already isomorphic on $(X_0,X_1)(X)$. 
Thus this map is equivalent to the inclusion
of a full subcategory of $\m(X)$ for each $X \in \cC$.

If $\m$ is not a stack, composing the map with a fibrant replacement for $\m$ shows that 
 $(X_0,X_1)(X) \to \m(X)$ is injective on morphisms and isomorphism classes.
\end{remark}
 
%\begin{note}
%Using the argument in the proof of Lemma \ref{ass-stack}
%one can also see that if $U \ra Y$ is a cover, given $Y \ra \m$,
%the induced map $(U, U \times^h_\m U) \ra (Y, Y \times^h_\m Y)$ 
%is a weak equivalence.
%\end{note}

\section{Representable Morphisms}

We begin by giving a definition of representable morphism in $P(\cC,\Gr)$
generalizing the one for stacks in \cite[Definition 4.2]{DM}
\footnote{ In \cite[3.9]{LM-B} for $\cC=\Aff_{\text{\'etale}}$ 
such morphisms are called schematic.}.  
Classically the definition of representable morphism applies only to
maps between stacks, for which the following two notions agree (by Lemma \ref{levelwisefp}).  

\begin{defn}
A morphism $\cM \ra \cN \in P(\cC, \Gr)$ is called 
\begin{itemize}
\item \dfn{strongly representable} 
if for each $X \in \cC$ and each map $X \ra \cN$, the 
homotopy fiber product $X \times^{h}_{\cN} \cM$ is levelwise weakly 
equivalent to a representable presheaf.
\item \dfn{representable} 
if for each $X \in \cC$ and each map $X \ra \cN$, the
homotopy fiber product $X \times^h_{\cN} \cM$ is weakly 
equivalent to a representable presheaf.
\end{itemize}
\end{defn}

\begin{note}
Note the following easy consequence of the homotopy invariance of the homotopy pullback:
If $f$ and $g$ are weakly equivalent morphisms in $P(\cC,\Gr)_L$, (i.e. there exist
$\alpha,\beta$ weak equivalences such that $\alpha \circ f = g \circ \beta$) 
then $f$ is representable if and only if $g$ is.
\end{note}

Representability allows one to extend certain properties of morphisms in $\cC$
to arbitrary presheaves of groupoids.
\begin{defn}
\label{extensionofproperties}
Let $P$ be a property of morphisms in $\cC$  
\footnote{In the usual definition it is also required that $P$ is local on the target
and stable under pullback (as in Definition \ref{localontarget}).  If $P$ is not stable
under pullback then property $P$ for representable functors will be a stablized version
of the origianl property.  For our purposes neither of these extra requirement makes a difference.}. 
We say $f\colon \cM \to \cN$ 
satisfies property P if for all maps $X \to \n$ with $X \in \cC$, the map 
$X \times_\cN^h \cM \to X $
is weakly equivalent to a map in $\cC$ which satisfies property P.

Similarly, a collection $\{ \cU_i \to \cN\}$ is a \emph{cover} if 
for each $X \to \m$ with $X \in \cC$, $\{ \cU_i\times^h_{\cN} X \to X \}$
is weakly equivalent to a cover in $\cC$.
\end{defn}
Notice that if $f$ satisfies property $P$ as above then it is 
necessarily representable.

Given a presheaf of groupoids $F$ recall that $\pi_0 F$ is the presheaf 
of groupoids defined by $(\pi_0 F)(X) = \pi_0 (F(X))$.

\begin{prop} \label{representable}
A map $f \colon \cM \ra \cN \in P(\cC,\Gr)$ is representable iff 
any fibration $p \colon \cM' \to \cN'$ weakly equivalent to $f$ is 
strongly representable.
In that case, for each map $X \to \m$, 
(the sheaf) $\pi_0(X \times_\m \m')$ is isomorphic to a representable.
\end{prop}
\begin{proof}
If $p$ is strongly representable, $f$ is obviously representable. For the converse
note that since $p$ is a fibration, given $X \to \cN'$, 
$X \times_{\cN'} \cM'$ is levelwise weakly equivalent
to the homotopy fiber product $X \times_{\cN'}^h \cM'$, which is by assumption 
weakly equivalent to a representable.  Now $X \times_{\cN'} \cM' \to X$
is a fibration, representables are fibrant and weak equivalences between 
fibrant objects are levelwise weak equivalences. Hence $X \times_{\cN}^h \cM$ is
levelwise weakly equivalent to a representable.

The second statement is clear.
\end{proof}

\begin{note}
The previous lemma shows that a map $f\colon \cM \to \cN$ is representable
iff the associated map of stacks is strongly representable.
\end{note}

\subsection{Generalized algebraic stacks}
In this section we define the concept of a generalized 
algebraic presheaf of groupoids.
We first recall the definition of algebraic stack which appears in 
\cite[4.6]{DM}.  This is usually called a Deligne-Mumford stack and 
we follow suite.  We also recall the weakening of this which is usually 
called an algebraic stack \cite[4.1]{LM-B}\footnote{Note that the definitions 
in \cite{LM-B} use a weakened form of representability which only
requires that the pullback be an algebraic space.}.

\begin{defn}  Let $S$ be a scheme and let $\cC$ be the category of 
$S$-schemes in the \'etale topology.
A stack $\m$ is called a Deligne-Mumford (resp. algebraic) stack
if the diagonal $\m \ra \m  \times \m$ is representable, 
separated and quasi-compact and if it admits an \'etale 
(resp. smooth) cover $X \ra \m$ with $X \in \cC$.
\end{defn}

\begin{defn} \label{ga}
Let $\cC$ be a site.  We say that $\m \in P(\cC,\Gr)$ 
is \emph{generalized algebraic} if its diagonal is representable and 
there is a cover $X \ra \m$ with $X \in \cC$.
\end{defn}

\begin{note}
The condition that the diagonal of $\m$ be representable
is equivalent to the requirement that for all $X \ra \m, Y \ra \m$,
with $X,Y \in \cC$ the product $X \times^h_\m Y$ is weakly equivalent to a 
representable.
\end{note}

\begin{lemma}
The definition of generalized algebraic is invariant under weak equivalence. 
Thus a presheaf of groupoids is generalized algebraic if
and only if its stackification is generalized algebraic.
\end{lemma}
\begin{proof}
If $X \to \m$ is a cover and $\n \to \m$ is a weak equivalence, 
the local lifting conditions \ref{locallifting} provide a cover of $\n$.
\end{proof}

\section{Characterization of the generalized algebraic stacks}

In this section we give a homotopy theoretic characterization of generalized
algebraic stacks (Theorem \ref{char-alg-stacks}).  For this we will need the following definition 
which generalizes faithfully flat descent of morphisms \cite[Theorem VIII.2.1]{SGA}.

\begin{defn}  
\label{descformorph}
Given a site $\cC$ consider the presheaf of groupoids
on $\cC$ defined on objects by $X \mapsto iso(\cC/X)$ and on morphisms via pullback.
We say that the site $\cC$ satisfies \emph{descent for morphisms} if
this is a stack.
\end{defn}

\begin{defn}
\label{localontarget}
We say that a topology on $\cC$ is \emph{local} if the notion of cover is local on the target. This means
that if $\{ U_i \to X\}$ is a cover and $\{ V_j \to X\}$ is a collection of morphisms such that 
$\{V_j \times_X U_i \to U_i\}$ is a cover for each $i$ then $\{ V_k \to X\}$ is also a cover.
\end{defn}

\begin{thm}
\label{char-alg-stacks}
Let $\cC$ be a site which is local and satisfies descent for morphisms. 
Then $\m$ is a generalized algebraic presheaf of groupoids if and only if 
$\m$ is weakly equivalent in $P(\cC,\Gr)_L$ to a groupoid object 
$(X_o, X_m)$ in $\cC$, for which the domain map $X_m \ra X_o$ is a cover.
\end{thm}
The proof is broken down into the following two propositions.
\begin{prop}
Let $\m$ be a generalized algebraic presheaf of groupoids and $X \ra \m$ be 
a cover (in the sense of Definition \ref{extensionofproperties}) with $X \in \cC$. 
Let $X_m$ denote the representable weakly equivalent to 
$X \times^h_{\m} X$.
Then the pair $(X, X_m)$ is a groupoid object in $\cC$ and the natural map 
$(X, X_m) \ra \m$ is a weak equivalence.
\end{prop}

\begin{proof}
It suffices to prove this for $\m$ a stack.
Given a generalized algebraic presheaf of groupoids there exists a 
representable morphism $X \ra \m$ which is a cover. 
Let $X \we \tilde{X} \fib \m$ be a factorization as a trivial cofibration
followed by a fibration, and let $\tilde{X}_\bullet$ denote
the nerve of this cover.

$\tilde{X}\times_\m X$ is levelwise weakly equivalent to 
a representable $X_m$ and $X_m \ra X$ is a cover.  
Since $(\tilde{X}, \tilde{X}\times_\m \tilde{X})$ is a groupoid object in 
the homotopy category so is $(X, X_m)$, and as
$X$ and $X_m$ are both fibrant, cofibrant, and discrete
$(X,X_m)$ is also a groupoid object in $P(\cC,\Gr)$ and in $\cC$.

Next we show that the map $(X,X_m)=|(X, X_m)_\bullet| \ra \m$ is a weak equivalence
by verifying that it satisfies the local lifting conditions.
The first of the local lifting conditions follows from the fact that
$X \ra \m$ is a cover.  By Remark \ref{incfullsubcat} the map 
$(X,X_m) \ra \m$ is levelwise equivalent to the inclusion of a 
full subcategory and so the second of the local lifting 
conditions is also satisfied (even not locally).
\end{proof}

\begin{prop} 
\label{descentformorph}
Let $\cC$ be a site which is local and satisfies descent for morphisms.
If $(X_o,X_m)$ is a groupoid object in $\cC$, with $X_m \ra X_o$ a 
cover then the associated stack 
$\cM_{(X_o,X_m)}$ is generalized algebraic.  
\end{prop}

\begin{proof}
First we will show that under these hypothesis $X_o \ra \m=\m_{(X_o,X_m)}$
is representable. 
%(This will follow from the hypothesis that $\cC$
%satisfies descent for morphisms.)

Let $Y \in \cC$ and $Y \ra \m$ be a map in $P(\cC,\Gr)$.
% and $X_o \times^h_\m Y$ the homotopy pullback.  
Since $X_o \ra \m$ is 
locally surjective there is a cover $U \ra Y$ for which we
have the following factorization

$$\xymatrix{U \times_Y U \ar@{-->}[d] \ar@2[r] & U \ar@{-->}[d] \ar[r] & Y \ar[d] \\
X_m \ar@2[r]  & X_o \ar[r] &  \m }$$
By construction of the homotopy fiber product we obtain a 
simplicial diagram of fibrations $U_\bullet \times_\m^h X_o \fib U_\bullet$
augmented by $Y\times^h_\m X_0 \fib Y$.
%$$\xymatrix{U \times_Y U \times_Y U \times^h_\m X_o \ar@3[r] \ar@{->>}[d] &
%U \times_Y U \times^h_\m X_o \ar@2[r] \ar@{->>}[d] & U \times^h_\m X_o 
%\ar@{->>}[d] \ar[r] & Y\times^h_\m X_o \ar@{->>}[d] \\
%U \times_Y U \times_Y U \ar@3[r] & U\times_Y U \ar@2[r] & U \ar[r] & Y}$$
Notice that besides $Y \times^h_\m X_o$ all of the fiber products 
$(U \times_Y U \dots \times Y_ U) \times^h_\m X_o$
are levelwise weakly equivalent to representables, for example
$$U\times_\m^h X_o=U\times_{X_o} X_o \times^h_\m X_o \we U \times_{X_o} X_m.$$
Fix $V \in \cC$ and $\alpha: U \times_\m^h X_o \we V$, and let $\bar{\alpha}$ be the 
induced isomorphism $\pi_0( U \times_\m^h X_o) \llra{\cong} V$.
Let $pr_1,pr_2$ denote the projections $U \times_Y U \ra U$.
It follows that we have weak equivalences
$$\pi_0( U \times_Y U \times^h_\m X_o) \llra{pr_i^*\bar{\alpha}} pr_i^*V,$$
and so we obtain an isomorphism over $U \times_Y U$,  
$$[(pr_2^*\bar{\alpha})^{-1} \circ pr_1^*\bar{\alpha}]: pr_1^* V \ra pr_2^* V$$
The simplicial identities imply that
this isomorphism satisfies the hypothesis for descent for morphisms, which implies
that there exists $V' \ra Y \in \cC$ 
together with an isomorphism $V\cong V'\times_Y U$
making the following diagram commute
$$\xymatrix{
U \times_Y U \times_Y U \times^h_\m X_o \ar@3[r] \ar[d]^\sim  &
U \times_Y U \times^h_\m X_o \ar@2[r] \ar[d]^\sim  & U \times^h_\m X_o 
\ar[d]^\sim  \ar[r] & Y\times^h_\m X_o  \ar@/^2pc/[dd]\\
U \times_Y U \times_Y U \times_Y V' \ar@3[r] \ar[d] &
U \times_Y U \times_Y V' \ar@2[r] \ar[d] & U \times_Y V' 
\ar[d] \ar[r] & V'  \ar[d] \\
U \times_Y U \times_Y U \ar@3[r] & U\times_Y U \ar@2[r] & U \ar[r] & Y}$$
It follows that $V'$ is weakly equivalent to
$$|U_\bullet \times_Y V'| \llla{\sim} |U_\bullet \times_\m X_o|\we Y \times^h_\m X_o.$$
Essentially the same argument implies that the diagonal 
$\m \ra \m \times \m$ is representable.  

To see that $X_o \ra \m$ is a cover, consider the pullback square
$$\xymatrix{ U \times_Y V' \cong U \times_{X_o} X_m \ar[r] \ar[d] &
V'  \cong Y \times^h_\cM X_o \ar[d] \\ U \ar[r] & Y.}$$
The bottom horizontal arrow and left vertical one are covers
since $X_m \ra X_o$ is a cover.  Since the 
topology on $\cC$ is local $V' \ra Y$ is a cover.
\end{proof}

By \cite[Expos\'e IX]{SGA}, we have the following corollary.
\begin{cor} Let $\cC =$ Schemes in the {\'e}tale topology. Then 
$\m$ is a Deligne-Mumford stack in the sense of \cite[Definition 4.6]{DM}
\footnote{In geometric situations it is usually also
required that the diagonal of $\m$ be quasi-compact and separated.} 
if and only if it is weakly equivalent to a groupoid object $(X_o,X_m)$ in $\cC$, with
 $X_m \ra X_o$ a cover.
\end{cor}

The flat topology on affine schemes satisfies descent for morphisms by 
\cite[Theorem VIII.2.1]{SGA} so we have the following corollary.
\begin{cor}
Let $\cC=$ Affine schemes in the flat topology. The generalized algebraic stacks
are those stacks which are weakly equivalent to flat Hopf algebroids.  
\end{cor}

\appendix

\section{Topologies on a stack}

We define two new sites $Rep/\m$ and $P(\cC,\Gr)/\m$ 
associated to a presheaf of groupoids $\m$.
We show that if $\m$ is a generalized algebraic stack then the category of sheaves $Sh(Rep/\m)$
agrees with the usual category of sheaves on $\m$.
In addition we prove a comparison theorem explaining the relation
between $Rep/\m$ and $P(\cC,\Gr)/\m$. 
\subsection{The site $\cC/\m$}

In this section we recall a site canonically associated to a presheaf of groupoids $\m$
first considered in \cite[Definition 4.10]{DM}.

\begin{defn}
\label{top-on-stack}
Let $\m$ be in $P(\cC,\Gr)$ and let $\cC/\m$ denote the 
site whose 
\begin{itemize} 
\item objects are pairs $(X,f)$, where $X\in \cC$ and $X \llra{f} \m$,
\item morphisms from $X \llra{f} \m$ to $X' \llra{g} \m$ are pairs 
$(h, \alpha)$ where $X \llra{h} X'$ and $\alpha$ is a homotopy 
$f \ra g \circ h$,
\item covers are collections of morphisms which forget to covers in $\cC$. 
\end{itemize}
\end{defn}

For a proof that this defines a Grothendieck topology see \cite[Section $2.1$]{H2}.

\begin{remark}
Given maps $f,f':X \ra \m$ a homotopy $\alpha: f \ra f'$ 
determines an isomorphism in $\cC/\m$ between the objects $f$ and $f'$.
So a presheaf $F$ on $\cC/\m$ will satisfy $F(X,f)\cong F(X,f')$. 
The category $\cC/\m$ is just the Grothendieck construction 
on the functor $\m$.
\end{remark}

\begin{remark}
Definition \ref{top-on-stack} generalizes the \'etale site \cite[4.10]{DM} of a Deligne-Mumford stack
which is the site defined above for
$\cC$ the category of schemes and \'etale maps in 
the \'etale topology.
However, there is no site $\cC$ which gives rise via \ref{top-on-stack}
to the smooth-\'etale site \cite[12.1]{LM-B} of an algebraic stack $\m$.
For example, if we take $\m$ to be a scheme, the smooth-\'etale site
is not the over category of $\m$ in some category of schemes which are 
the only kind of sites which arise via \ref{top-on-stack}.
\end{remark}

\subsection{The site $Rep/\m$}

The concept of representable morphism allows us to extend in a 
natural way the notion of cover to presheaves of groupoids
and so gives rise to the following site.

\begin{defn} 
\label{defnrepm}
For $\cM$ in $P(\cC,\Gr)$ the site $Rep/\cM$ has
\begin{itemize}
 \item Objects: strongly representable morphisms $\cN \fib \cM$, 
 \item Morphisms from $\cN_1 \llra{f_1} \m$ to $\cN_2 \llra{f_2} \m$ 
 consist of pairs $(g,\alpha)$ with $g:\cN_1 \ra \cN_2$ and $\alpha$ 
 a homotopy $f_1 \ra f_2 \circ g$.
 \item Covers: collections of morphisms $\{\cN_i \xrightarrow{u_i} \cN\}$
		such that the $u_i$ are strongly representable 
		and for each $X \ra \cN, X \in \cC$ the collection 
		$\{ X \times_{\cN}^h \cN_i \ra X \}$
		is weakly equivalent to a cover in $\cC$. 
\end{itemize} \end{defn}
Note that the pullback in $Rep/\m$ is exactly the homotopy fiber product of
Definition \ref{fiber-product}. It is also true that homotopy equivalences 
are isomorphisms in $Rep/\m$. The proof that $Rep/\m$ is a site is parallel 
to that for $\cC/\m$.

\begin{prop} 
Let $\cM$ be a generalized algebraic stack
then the category of sheaves on $Rep/\cM$ 
is equivalent to the category of sheaves on the site $\cC/\cM$.
\end{prop}
 
\begin{proof}
Since $\cM$ is a generalized algebraic stack, $\cC/\m$ embeds in $Rep/\m$
as a full subcategory. By \cite[Proposition 3.9.1]{Ta} it is enough to see 
that any object in $Rep/\m$ is covered by an object in $\cC/\m$.

Given an object $f: \n \to \m$ in $Rep/\m$, and a cover $X \to \m$
with $X \in \cC$, $X\times^{h}_\m \n$ is levelwise weakly equivalent
to a representable $Z \cong \pi_0(X\times^{h}_\m \n)$. The quotient
map $p \colon X \times^{h}_\m \n \to Z$ is a trivial fibration and $Z$ is cofibrant
so $p$ is a homotopy equivalence and hence an isomorphism in $Rep/\m$.
It follows that $Z \to X\times^{h}_\m \n \to \n$ is a cover in $Rep/\m$.
\end{proof}

\begin{note}
One can make a definition analogous to Definition \ref{defnrepm} using the 
concept of representable instead of strongly representable morphism, but then the
result of the previous proposition would not hold as local weak equivalences
would not be isomorphisms in the category.
\end{note}

\begin{remark} 
Let $\{ f_i \colon \cU_i \to \cN\}$ be a collection of representable morphisms
and $\{\tilde{f_i}: \tilde{\cU_i} \fib \cN \}$ be the family of fibrations 
obtained by functorial factorization in $P(\cC,\Gr)_L$. Then the following are
equivalent:
\begin{enumerate}[(i)] 
\item The collection $\{f_i \colon \cU_i \to \cN \}$ is a cover in the sense of Definition \ref{extensionofproperties}.
\item The collection $\{\tilde{f_i}: \tilde{\cU_i} \fib \cN \}$ is a cover in $Rep/\m$.
\end{enumerate}
\end{remark}

\subsection{The site $P(\cC,\Gr)/\m$} We now define a site associated to $\m$ 
which is very natural from the point of view of homotopy theory and compare it 
to the ones discussed above.

\begin{thm} \label{new-GT}
Let $\cC$ be a site and $\m \in P(\cC,\Gr)_L$.  Then there is a 
Grothendieck topology on $P(\cC,\Gr)_L/\m$ in which the covers are
the sets of morphisms $\{ \cU_i \ra \cN\}$ which satisfy:
\begin{itemize}
	\item $\cU_i \fib \cN$ are fibrations,
	\item $|\cU_\bullet| \ra \cN$ is a weak equivalence.
\end{itemize}
\end{thm}

\begin{proof} 
First we prove that pullbacks of covers are covers.
Let $\coprod \cU_i \ra \cN$ be a cover and $\m \ra \cN$ a morphism.
The morphism $\coprod \cU_i \ra \cN$ is an objectwise fibration and so 
the induced map $|\cU_\bullet| \ra \cN$ is also an objectwise fibration. 
As geometric realization commutes with fiber products, 
$|\cU_\bullet| \times_{\cN} \m \cong |\cU_\bullet \times_{\cN} \m|$ 
and so we have a pullback square
$$\xymatrix{|\cU_\bullet \times_{\cN} \m| \ar[r] \ar[d] & 
|\cU_\bullet| \ar[d] \\ \m \ar[r] & \cN}$$
where the right vertical map is an objectwise fibration and a weak equivalence.
By \cite[Corollary 5.8]{H} the pullback of a weak equivalence 
which is an objectwise fibration is a weak equivalence.

To see that covers compose, let $\{\cV_{ij} \ra \cU_i\}$ be covers 
of each $\cU_i$. 
The iterated fiber products of the covers $\{\cV_{ij} \ra \cU_i\}$
form a bisimplicial object $\cV_{\bullet,\bullet}$ augmented over $\cU_\bullet$. 
The columns $\cV_{n,\bullet}$ are iterated fiber products of the nerves
of $\{\cV_{ij} \ra \cU_i\}$ and therefore the map induced by the
augmentation
\[ |\cV_{\bullet,\bullet}| \to |\cU_\bullet| \to \cN \]
is a weak equivalence.
The geometric realization of the bisimplicial object is equivalent 
to the geometric realization of its diagonal, so $|\diag \cV_{\bullet,\bullet}| \ra \cN$
is a weak equivalence. 

The nerve of the cover $\{\cV_{ij} \to \cN\}$ is the 0-th row $\cV_{\bullet,0}$
Since $\cV_{\bullet,0}$ is a 0-coskeleton over $\cN$, there 
is a retraction to the canonical map $\cV_{\bullet,0} \to \diag\cV_{\bullet,\bullet}$ over $\cN$
(see \cite[Proposition A4]{DHI}) and therefore $ \cV_{\bullet,0} \to \cN $
is a weak equivalence.
\end{proof}

\begin{prop} \label{charac-covers}
Given a collection $\{ f_i \colon \cU_i \fib \cN\}$ in $P(\cC,\Gr)/\m$  
the following are equivalent:
\begin{enumerate}[(i)]
\item The collection $\{ f_i \colon \cU_i \fib \cN\}$ is a cover in $P(\cC,\Gr)/\m$.
\item For each $X \ra \cN$ the collection $\{\cU_i\times_{\cN} X \fib X \}$ is 
	a cover in $P(\cC,\Gr)/\m$.
\item  For each $X \ra \cN$ the union of the images of $\cU_i\times_{\cN} X \fib X $ 
	is a covering sieve of $X$ in $\cC$. 
\end{enumerate}
\end{prop}

\begin{proof} 
The fact that (i) implies (ii) is a part of the axioms for a topology.
First we prove that that (ii) implies (i).  Given
% are equivalent With out loss of generality we may assume that 
%$\{ f_i \colon \cU_i \to \cN \}$ are fibrations.
$X \in \cC,$ and $X \ra \cN$, let $W$ denote $|\cU_{\bullet}|$,
then the projection map $W\times_{\cN}X \ra X$
%the pullback (which in this case is 
%equivalent to the homotopy pullback) is a cover of $X \in P(\cC,\Gr)_L$. 
%implies that for each $X \in \cC$, and $X \ra \cN$, we have the 
%\[ \xymatrix{ 
%W \times_{\cN} X \ar[r] \ar[d]^{\sim} & W \ar[d] \\
%X \ar[r] & \cN.}
%\]
%Here the left vertical arrow 
is a weak equivalence since
$$W \times_{\cN} X \cong |\cU_{\bullet}\times_{\cN} X| \cong 
|(\cU \times_{\cN} X)_{\bullet}|,$$ 
so this map is the induced map to $X$ from the nerve of the 
cover $\cU_i \times_{\cN} X \ra X$.

Similarly, given any map $X \otimes \Delta^1 \ra \cN$ the 
pullback $(X \otimes \Delta^1) \times_\cN W \ra X \otimes \Delta^1$
is a weak equivalence since in the diagram
\[ \xymatrix{ W \times_{\cN} X \ar[r] \ar[d]^\sim &
W \times_{\cN} (X \otimes \Delta^1) \ar[r] \ar[d] & W  \ar[d] \\ 
X \ar[r]^\sim & X \otimes \Delta^1 \ar[r] & \cN}  \]
the top left map is a levelwise weak equivalence (because $W \to \n$
is a levelwise fibration and $\Gr$ is right proper).
It is now straightforward to check that $W \ra \cN$ is a weak equivalence
using the local lifting conditions \ref{locallifting}.

%The first of the local lifting conditions holds since for any $X \ra \cN$
%the cover $\cU_i \times_\cN X \ra X$ itself solves the lifting problem
%(with no isomorphism required):
%\[ \xymatrix{ \coprod (\cU_i \times_\cN X) \ar[r] & 
%W \times_{\cN} X \ar[r] \ar[d]^{\sim} & W \ar[d] \\
%&  X \ar[r] & \cN .} \]
%To check the second condition, given 
%\[ \xymatrix{ X \coprod X \ar[r] \ar[dr] & 
%W \times_\n (X \otimes \Delta^1) \ar[r] \ar[d]^{\sim} 
%& W \ar[d] \\  & X \otimes \Delta^1 \ar[r] & \cN,} \]
%We have a cover $U \ra X$, (where $U$ denote the sheaf $\coprod U_i$), 
%such that a lift exists in 
%\[ \xymatrix{ U \coprod U  \ar[d] \ar[r] & X \coprod X \ar[r] & 
%W \times_\cN (X \otimes \Delta^1) \ar[r] \ar[d]^{\sim} & W \ar[d] \\  
%U \otimes \Delta^1 \ar[rr] \ar@{-->}[urr] & &  X \otimes \Delta^1 \ar[r] & 
%\cN .} \]
%Finally to check the remaining condition, given the diagram 
% \[ \xymatrix{ U \otimes B\Z  \ar[d] \ar[r] & X \otimes B\Z \ar[r] & 
%W \times_\cN X  \ar[r] \ar[d]^{\sim} & W \ar[d] \\  
%U \ar[rr] \ar@{-->}[urr] & &  X \otimes \Delta^1 \ar[r] & \cN,} \]
%we can obtain the dotted lift just as before.

To see that (ii) implies (iii) let $\{\cU_i \fib X \}$ be a cover in $P(\cC,\Gr)/\m$
and let $F$ be any sheaf on $\cC$. $F$ is a discrete stack and so  
$$Map(X,F) \we \holim Map(\cU_{\bullet}, F) \cong \lim Map(\cU_{\bullet}, F) 
\cong Map(\colim \pi_0\cU_{\bullet}, F)$$
which shows that $X$ is the coequalizer of the sheafification 
$(\coprod \pi_0\cU_{ij} \dbra \coprod \pi_0\cU_i)$, so by \cite[Corollary III.7.7]{MM}
the union of the images of $\pi_0\cU_i \ra X$ is a covering sieve in $\cC$.

Conversely suppose that $\{\cU_i \fib X \}$ generates a covering sieve.
This means that $sh(\coprod \pi_0\cU_i) \ra X$ is a surjection of sheaves, from 
which it follows that $\colim \pi_0 \cU_\bullet \ra X$ is a weak equivalence.
Since $\cU_\bullet$ is a $0$-coskeleton in simplicial objects over $X$
the projection $|\cU_\bullet| \ra \colim \pi_0 \cU_\bullet$ is a levelwise weak
equivalence.  It follows that $\{ \cU_i \fib X \}$ 
is a cover in $P(\cC,\Gr)/\m$.

\end{proof}

Here is the relation between the notion of cover on $P(\cC,\Gr)/\m$ just defined
with the ones defined previously.

\begin{cor} 
Let $\{ f_i \colon \cU_i \to \cN\}$ be a collection of representable morphisms
and $\{\tilde{f_i}: \tilde{\cU_i} \fib \cN \}$ be fibrations 
obtained by functorial factorization in $P(\cC,\Gr)_L$. 
\begin{enumerate}
\item If $\{f_i \colon \cU_i \to \cN \}$ is a cover in the sense of Definition 
	\ref{extensionofproperties}, then $\{\tilde{f_i}: \tilde{\cU_i} \fib \cN \}$ 
	is a cover in $P(\cC,\Gr)/\m$.
\item Conversely, if $\{\tilde{f_i}: \tilde{\cU_i} \fib \cN \}$ 
	is a cover in $P(\cC,\Gr)/\m$ then for each $X \ra \cN$ the collection
	$\{\cU_i\times^h_{\cN} X \ra X \}$ determine a covering sieve in $\cC$.
\end{enumerate}
\end{cor}

Let $Sh(\m)$ be the category of sheaves on $P(\cC,\Gr)_L/\m$ which take 
weak equivalences to isomorphisms.  
The above Corollary implies that we have a surjective restriction functor $Sh(\m) \ra Sh(\cC/\m)$.

\end{document}